\documentclass[12pt, reqno]{amsart}
\usepackage{amsthm}
\usepackage{amsmath,amssymb,amsthm,mathrsfs,amscd}
\usepackage{MnSymbol}
\usepackage{color}
\usepackage[all]{xy}
\usepackage[english]{babel}
\usepackage{graphicx}
\usepackage[latin1]{inputenc}
\usepackage[T1]{fontenc}
\usepackage[final=true]{hyperref}
\usepackage{tikz}
\usetikzlibrary{patterns}
\usetikzlibrary{arrows}
\usepackage{rotating}
\usepackage{cases}
\usepackage{xkeyval}
\usepackage{faktor}

\newtheorem{theorem}{Theorem}
\newtheorem{theorem*}{Theorem}
\newtheorem{lemma}{Lemma}
\newtheorem{proposition}{Proposition}
\newtheorem{corollary}{Corollary}

\newtheorem{remark}{Remark}
\newtheorem{conjecture}{Conjecture}
\newtheorem{definition}{Definition}

\newcommand{\ii}{{\bf i}}
\newcommand{\Hom}{Hom}
\newcommand\restr[2]{{
  \left.\kern-\nulldelimiterspace 
  #1 
  \vphantom{\big|} 
  \right|_{#2} 
  }}
  \makeatletter
\DeclareRobustCommand*{\mfaktor}[3][]
{
   { \mathpalette{\mfaktor@impl@}{{#1}{#2}{#3}} }
}
\newcommand*{\mfaktor@impl@}[2]{\mfaktor@impl#1#2}
\newcommand*{\mfaktor@impl}[4]{
   \settoheight{\faktor@zaehlerhoehe}{\ensuremath{#1#2{#3}}}%
   \settoheight{\faktor@nennerhoehe}{\ensuremath{#1#2{#4}}}%
      \raisebox{-0.00\faktor@zaehlerhoehe}{\ensuremath{#1#2{#3}}}%
      \mkern-4mu\diagdown\mkern-5mu%
      \raisebox{0.00\faktor@nennerhoehe}{\ensuremath{#1#2{#4}}}%
}
\newcommand{\sym}[2]{\{ #1 , #2\}}

\begin{document}

	\title[Redundancy \& multiplicities]{Redundancy in string cone inequalities and multiplicities in potential functions on cluster varieties}

	\author{Gleb Koshevoy}
	\address{Institute of Information Transmission Problems Russian Academy of Sciences, National Research University Higher School of Economics, Russian Federation}
	\email{koshevoyga@gmail.com}
	
	\author{Bea Schumann}
	\address{Mathematical Institute, University of Cologne}
	\email{bschuman@math.uni-koeln.de}

	\begin{abstract} 
		We study defining inequalities of string cones via a potential function on a reduced double Bruhat cell. We give a necessary criterion for the potential function to provide a minimal set of inequalities via tropicalization and conjecture an equivalence.
		\end{abstract}
	
	\keywords{Cluster Algebras; Quantum Groups; Canonical Bases; String Polytopes}
	
	\maketitle
	\section*{Introduction}
Let $\mathfrak{g}$ be simple simply laced complex Lie algebra. To every reduced expression $\ii$ of the longest element $w_0$ of the Weyl group of $\mathfrak{g}$ is associated a polyhedral cone $\mathcal{C}_{\ii}\subset \mathbb{R}^N$ called the string cone. Here $N$ is the length of $w_0$. The string cones arise in many different contexts, e.g. they are closely connected to the dual canonical basis of the universal enveloping algebra of the negative part of $\mathfrak{g}$ (\cite{BZ93}), may be seen as generalisations of Gelfand-Tsetlin cones (\cite{Lit}), play a great role in toric degenerations of flag varieties (\cite{FFL}) and are closely related to tensor product multiplicities of representations of $\mathfrak{g}$ (\cite{BZ}).

This paper deals with the problem of determining a minimal set of inequalities for $\mathcal{C}_{\ii}$, i.e. describing the facets of $\mathcal{C}_{\ii}$. The way we approach this is as follows. If $\mathbf{j}$ is another reduced word of $w_0$ there is a piecewise linear bijection $\Psi_{\mathbf{j}}^{\mathbf{i}}:\mathcal{C}_{\mathbf{i}} \rightarrow \mathcal{C}_{\mathbf{j}}.$ The cone $\mathcal{C}_{\ii}$ consists of all $t\in \mathbb{R}^N$ such that for any reduced word $\ii'$ the last coordinate of $\Psi_{\mathbf{j}}^{\mathbf{i}}(t)$ is non-negative. This fact was used in \cite{GKS20} to define a function $\varsigma_{\ii}\in \mathbb{C}[x^{\pm}_j\mid j \in \{1,2,\ldots,N\}]$ (see Definition \ref{varsigma}) such that $\mathcal{C}_{\ii}$ is given by all $t\in \mathbb{R}^N$ with $[\varsigma_{\ii}]_{trop}(t)\ge 0$. Here trop means the tropicalization as defined in Section \ref{trop}. We prove the following sufficient criterion.

\begin{theorem}[Theorem \ref{nomulti}] If the exponents of all variables in $\varsigma_{\mathbf{i}}$ have absolute value less or equal to $1$, then the set of inequalities given by $[\varsigma_{\mathbf{i}}]_{trop}$ is non-redundant.
\end{theorem}

We can refine the above thanks to the observation that the Laurent-polynomial $\varsigma_{\mathbf{i}}$ naturally splits into a sum of Laurent polynomials $\varsigma_{\mathbf{i}}=\sum_{i\in I}\varsigma_{\mathbf{i},i}$ where $I$ is the index set of simple roots of $\mathfrak{g}$. 

\begin{proposition}[Proposition \ref{stringred}] The inequalities arising from $[\varsigma_{\mathbf{i},i}]_{trop}$ and $[\varsigma_{\mathbf{i},j}]_{trop}$ are independent, i.e. the set of inequalities given by $[\varsigma_{\mathbf{i}}]_{trop}$ is non-redundant if and only if the sets inequalities given by $[\varsigma_{\mathbf{i},i}]_{trop}$ are non-redundant for all $i\in I$. 
\end{proposition}

Proposition \ref{stringred} allows us to study each summand $\varsigma_{\mathbf{i},i}$ independently. We call $\varsigma_{\mathbf{i},i}$ \emph{multiplicity-free} if the exponents of all variables have absolute value less or equal to $1$ and thus the monomials of $\varsigma_{\mathbf{i},i}$ correspond to facets of the string cone via tropicalization. 

We provide several classes of examples of reduced words $\ii$ where $\varsigma_{\mathbf{i},i}$ is multiplicity-free. In Theorem \ref{simbraid} we show that this is the case if $\ii$ is simply-braided for $i$, a notion that was studied in \cite{SST}. Furthermore $\varsigma_{\mathbf{i},i}$ is multiplicity-free for all $i\in I$ if $\ii$ is a nice word in the sense of Littelmann \cite{Lit} by Theorem \ref{nicew}. 

In the case that $\omega_i$ is a minuscule weight of $\mathfrak{g}$ we show that $\varsigma_{\mathbf{i},i}$ is given via Berenstein-Zelevinsky's $\mathbf{i}$-trails in Section \ref{trails}. Theorem \ref{minusculenomulti} proves that $\varsigma_{\mathbf{i},i}$ is again multiplicity-free in that case. Thus, in particular, $\varsigma_{\mathbf{i},i}$ is always multiplicity-free for $\mathfrak{g}$ of type $A$. 

We conjecture that Theorem \ref{nomulti} is indeed also a necessary criterion.

\begin{conjecture}[Conjecture \ref{conjmu2}] The inequalities arising from $[\varsigma_{\mathbf{i},i}]_{trop}$ are non-redundant if and only if $\varsigma_{\mathbf{i},i}$ is multplicity-free.
\end{conjecture}

Finally in Section \ref{exnonm} we provide an example for which $\varsigma_{\mathbf{i},i}$ is not multiplicity-free. In this example there is exactly one Laurent monomial of $\varsigma_{\mathbf{i},i}$ whose tropicalization leads to a redundant inequality. We note that this monomial is the only monomial of $\varsigma_{\mathbf{i},i}$ with a coefficient greater than $1$. This leads us to the final conjecture given a criterion to determine the facets of the string cone.

\begin{conjecture}(Conjecture \ref{mult2}) An inequalities arising from the tropicalization of a monomial of $\varsigma_{\mathbf{i},i}$ is redundant if and only if the coefficient of this monomial is greater than $1$.
\end{conjecture}

Note that this criterion cannot be seen merely in the tropicalization of $\varsigma_{\mathbf{i},i}$ since coefficients do not play any role there. However, it is visible in the tropicalization wether $\varsigma_{\mathbf{i},i}$ is multiplicity-free or not. It would be very interesting to find a conceptual explanation of our result. In Remark \ref{ffei} we suggest a relation to $F$-polynomials of cluster variables.

The proofs of Theorem \ref{nomulti} and Proposition \ref{stringred} are making use of the following fact proven in \cite{GKS20}. The function $\varsigma_{\mathbf{i}}$ is the pullback of the potential function of \cite{GHKK} on the big reduced double Bruhat cell (see Proposition \ref{iso}) by an isomorphism of tori. An important point here is that the potential function (expressed in appropriate torus coordinates) only has non-positive exponents (Proposition \ref{polynomial}). The results may be obtained from this using cluster combinatorics.

The paper is organized as follows. The first section introduces important notion related to reduced words and tropicalization. The second section deals with string cones and their defining inequalities and states our main results. The third section recalls the notion of $\mathcal{A}$- and $\mathcal{X}$-cluster varieties and potential functions in the sense of \cite{GHKK}. In Section $4$ we restrict ourself to the big reduced double Bruhat cell as specific cluster variety and recall how to obtain string cone inequalities via potential functions. The proofs of Theorem \ref{nomulti} and Proposition \ref{stringred} are obtained in the fifth section. The sixth section provides examples of reduced words for which our system of inequalities is non-redundant. In the final Section $7$ we provide an example of a redundancy corresponding to a multiplicity in $\varsigma_{\mathbf{i}}$.

\section*{Acknowledgement}
We would like to thank Volker Genz for insightful discussions. G. Koshevoy thanks the grant RSF 21-11-00283 for support. B. Schumann was supported by the SFB/TRR 191 'Symplectic Structures in Geometry, Algebra and Dynamics', funded by the DFG.

\section{Background}

\subsection{Notation}
For a positive integer $m\in \mathbb{Z}_{\ge 0}$ we denote by $[m]$ the set $\{1,2,\ldots,m\}$. Let $\mathfrak{g}$ be simple simply laced complex Lie algebra of rank $n$, $I:=[n]$, $C=(c_{i,j})_{i,j\in I}$ its Cartan matrix and $\mathfrak{h}\subset \mathfrak{g}$ a Cartan subalgebra.
We choose simple roots $\Delta^+=\{ \alpha_i \}_{i\in I}\subset \mathfrak{h}^*$ and simple coroots $\{\alpha^{\vee}_i\}_{i\in I}\subset\mathfrak{h}$ with $\alpha_i (\alpha^{\vee}_j)= c_{j,i}$. We denote by $\Delta^+\subset\mathfrak{h}^*$ the set of positive roots and by $\Delta$ the set of roots associated to $\{\alpha_i \mid i\in I \}$.

The fundamental weights $\{\omega_i\}_{i\in I}\subset\mathfrak{h}^*$ of $\mathfrak{g}$ are given by $\omega_i(\alpha^{\vee}_j)=\delta_{i,j}$. We denote by $P=\langle \omega_i \mid i\in[n] \rangle_{\mathbb{Z}}$ the weight lattice of $\mathfrak{g}$ and by $P^+=\langle \omega_i \mid i\in[n] \rangle_{\mathbb{Z}_{\ge 0}}\subset P$ the set of dominant weights.

The Langlands dual Lie algebra $^L\mathfrak{g}$ of $\mathfrak{g}$ is the simple, simply laced complex Lie algebra with Cartan matrix $C$, Cartan subalgebra $\mathfrak{h}^*$, simple roots $\{\alpha^{\vee}_i\}_{i\in I}$, simple coroots $\{\alpha_i\}_{i\in I}$ and $\alpha^{\vee}_i (\alpha_j)=c_{i,j}$.
The fundamental weights of $^L\mathfrak{g}$ are $\{\omega^{\vee}_i\}_{i\in I}\subset \mathfrak{h}$ where $\alpha_i (\omega^{\vee}_j) = \delta_{i,j}$.

\subsection{Weyl groups and reduced words}

The Weyl group $W$ of $\mathfrak{g}$ is a Coxeter group generated by the simple reflections $s_i$ ($i \in I$) with relations
\begin{align*} s_i^2 &=id, \\
 s_{i_1}s_{i_2} & =s_{i_2}s_{i_1} \qquad \, \, \, \text{if }c_{i_1,i_2}=0 \quad \text{ ($2$-term relation)},  \\
s_{i_1}s_{i_2}s_{i_1}&=s_{i_2}s_{i_1}s_{i_2} \quad \text{ if }c_{i_1,i_2}=-1 \quad \text{ ($3$-term relation)}.
\end{align*}

We sometimes call a $2$-term relation also a commutation relation.

The group $W$ has a unique longest element $w_0$ of length $N=\#\Delta^+$, where the length is given by the minimal number of generators in an expression. For a reduced expression $s_{i_1} \cdots s_{i_N}$ of $w_0$, i.e. an expression of minimal length, we write $\ii:=(i_1,\ldots, i_N)$ and call $\ii$ a \emph{reduced word} (for $w_0$). The set of reduced words for $w_0$ is denoted by $R(w_0)$.

The group $W$ acts on $P$ as $s_i(\lambda)=\lambda-\lambda(\alpha^{\vee}_i)\alpha_i$ for $\lambda \in P$ and $i\in I$. We denote by $i^*\in I$ the unique element such that $w_0\omega_i=-\omega_{i^*}$.

We have two operations on the set of reduced words $R(w_0)$.

A reduced word $\mathbf{j}=(j_1,\ldots,j_N)$ is defined to be obtained from $\ii=(i_1,\ldots,i_N)\in R(w_0)$ by a \emph{$2$-term move at position $k\in [N-1]$} if $i_{\ell}=j_{\ell}$ for all $\ell\notin \{k,k+1\}$, $(i_{k+1}, i_{k})=(j_{k},j_{k+1})$ and $c_{i_{k},i_{k+1}}=0$.

A reduced word $\mathbf{j}=(j_1,\ldots,j_N)$ is defined to be obtained from $\ii=(i_1,\ldots,i_N)\in R(w_0)$ by a \emph{$3$-term move at position $k\in [N-1]$} if $i_{\ell}=j_{\ell}$ for all $\ell\notin \{k-1,k,k+1\}$, $j_{k-1}=j_{k+1}=i_k$, $j_k=i_{k-1}=i_{k+1}$ and $c_{i_{k},i_{k+1}}=-1$.

By the Tits theorem every two reduced words for an element $w\in W$ can be obtained from each other by a sequence of $2$-term and $3$-term moves.

We call a total order $\le$ on $\Delta^+$ \emph{convex} if for any two positive roots $\beta_1,\beta_2$ such that $\beta_1+\beta_2\in \Delta^+$, we either have $\beta_1 < \beta_1+\beta_2 < \beta_2$ or $\beta_2 < \beta_1+\beta_2 < \beta_1$.
By \cite[Theorem p. 662]{P94} the set of reduced words is in bijection with the set of convex orders as follows. For a reduced word $\ii=(i_1,\ldots,i_N)\in R(w_0)$ the total order \begin{equation*}
\alpha_{i_1}<_{\ii} s_{i_1}(\alpha_{i_2}) <_{\ii} \ldots <_{\ii} s_{i_1}\cdots s_{i_{N-1}}(\alpha_{i_N})
\end{equation*}
on $\Delta^+$ is convex and every convex order on $\Delta^+$ arises that way. We write $\Delta^+_{\ii}=\{\beta_1,\beta_2,\ldots,\beta_N\}$ for the set of positive roots ordered with respect to the convex order $<_{\ii}$ and identify $\Delta_{\ii}^+$ with $[N]$ via
\begin{equation}\label{posident}
\beta_k \mapsto k.
\end{equation}
\subsection{Tropicalization}\label{trop}
We recall the notion of tropicalization from \cite{GHKK}. Let $\mathbb{G}_m$ be the multiplicative group. Let  $\mathbb{T}=\mathbb{G}_m^k$ be an algebraic torus. We denote by $[\mathbb{T}]_{trop}=\Hom(\mathbb{G}_m,\mathbb{T})=\mathbb{Z}^k$ its cocharacter lattice. A positive (i.e. subtraction-free) rational map $f$ on $\mathbb{T}$, $f(x)=\frac{\sum_{u\in I} a_u x^u}{\sum_{u\in J} b_u x^u}$ with {$a_u, b_u \in \mathbb{R_{+}}$}, gives rise to a piecewise-linear map
\begin{equation*}
[f]_{trop} : [\mathbb{T}]_{trop} \rightarrow [\mathbb{G}_m]_{trop}=\mathbb{Z}, \quad x\mapsto \min_{u \in I} \left\langle x,u\right\rangle - \min_{u\in J} \left\langle x,u\right\rangle,
\end{equation*}
where $\langle \cdot ,\cdot \rangle$ is the standard inner product of $\mathbb{Z}^k$. We call $[f]_{trop}$ the \emph{tropicalization} of $f$. 

For a positive rational map
$$f=(f_1, \dots, f_\ell) : \mathbb{G}_m^k \dashedrightarrow \mathbb{G}_m^{\ell}$$
we define its tropicalization as 
$$[f]_{trop}:=([f_1]_{trop}, \ldots ,[f_{\ell}]_{trop}): [\mathbb{G}_m^k]_{trop} \rightarrow [\mathbb{G}_m^{\ell}]_{trop}.$$


\section{String cones}
\subsection{String parametrization of the canonical basis}

Let $B(\infty)$ be the crystal basis of $U_q^-$ in the sense of \cite{Ka94} with partially inverse crystal operators $\tilde{e}_i$, $\tilde{f}_i$ for $i\in I$. For each reduced word $\mathbf{i}=(i_1,i_2,\ldots,i_N)\in R(w_0)$ we define the $\mathbf i$-string datum of an element of $B(\infty)$ as follows. 

\begin{definition} 
An \emph{$\mathbf i$-string datum} $\text{str}_{\mathbf i}(b)$ of $b\in B(\infty)$ is defined by a tuple $(x_1,x_2,\ldots,x_N)\in \mathbb{Z}_{\ge 0}^{\Delta_{\ii}^+}$ determined inductively by

\begin{align*}
x_1 &= \displaystyle\max_{ k \in \mathbb{Z}_{\ge 0}} \{\tilde{e}_{i_1}^k b \in B(\infty)\}, \\
x_2 &= \displaystyle\max_{ k \in \mathbb{Z}_{\ge 0}} \{ \tilde{e}_{i_2}^{k}\tilde{e}_{i_1}^{x_1} b \in B(\infty)\}, \\
& \vdots \\
x_N &= \displaystyle\max_{ k \in \mathbb{Z}_{\ge 0}} \{ \tilde{e}_{i_N}^k \tilde{e}_{i_{N-1}}^{x_{N-1}}\cdots \tilde{e}_{i_1}^{x_1} b \in B(\infty) \}.
\end{align*}

By \cite{BZ,Lit} the subset 
$$\mathcal{C}_{\mathbf{i}}:=\{\text{str}_{\mathbf i}(b) \mid b \in B(\infty)\}\subset \mathbb{Z}_{\ge 0}^{\Delta_{\ii}^+}$$ 
is a polyhedral cone called the \emph{string cone associated to $\mathbf i$}. 
\end{definition}

By \cite{BZ, Lit}:

\begin{definition}\label{eq:strans}  Let $\mathbf{j} \in R(w_0)$ be obtained from $\mathbf{i} \in R(w_0)$ by a $3$-term move at position $k$. We have a piecewise linear bijection
\begin{align*} \Psi_{\mathbf{j}}^{\mathbf{i}}:\mathcal{C}_{\mathbf{i}} & \rightarrow \mathcal{C}_{\mathbf{j}} \\
(x_{\ell})_{\beta_{\ell}\in \Delta^+_{\mathbf{i}}} & \mapsto (x'_{\ell})_{\beta_{\ell}\in \Delta^+_{\mathbf{j}}}
\end{align*}
given by
\begin{equation*}
x'_{k-1}=\max{(x_{k+1},x_{k}-x_{k-1})}, \quad x'_k=x_{k-1}+x_{k+1}, \quad x'_{k+1}=\min{(x_{k-1},x_k-x_{k+1})}.
\end{equation*}
and $x'_{\ell}=x_{\ell} \quad \forall \ell \notin \{k-1,k\}$.

Let $\mathbf{j}\in R(w_0)$ be obtained from $\mathbf{i}\in R(w_0)$ by a $2$-term move at position $k$. We have a linear bijection
\begin{align*} \Psi_{\mathbf{j}}^{\mathbf{i}}:\mathcal{C}_{\mathbf{i}} & \rightarrow \mathcal{C}_{\mathbf{j}} \\
(x_{\ell})_{\beta_{\ell}\in \Delta^+_{\mathbf{i}}} & \mapsto (x'_{\ell})_{\beta_{\ell}\in \Delta^+_{\mathbf{j}}}
\end{align*}
given by
\begin{equation*}
x'_{k}=x_{k+1}, \quad x'_{k+1}=x_{k}, \quad x'_{\ell}=x_{\ell} \quad \forall \ell \notin \{k-1,k\}.
\end{equation*}
Let $\mathbf{j},\mathbf{i} \in R(w_0)$ by two arbitrary reduced words. We define $\Psi_{\mathbf{j}}^{\mathbf{i}}:\mathcal{C}_{\mathbf{i}}  \rightarrow \mathcal{C}_{\mathbf{j}}$ to be the composition of the above defined bijections corresponding to a sequence of $2-$ and $3-$moves transforming $\mathbf{i}$ into $\mathbf{j}$.
\end{definition}

\subsection{Inequalities of string cones}
\begin{definition}\label{varsigma} Let $\mathbf{i},\mathbf{j}\in R(w_0)$ and $i\in I$. We define $\varsigma_{\mathbf{i},i}$ to be the rational function on $\mathcal{T}_{\mathbf{i}}=(\mathbb{C}^*)^{\Delta^+_{\mathbf{i}}}$ uniquely determined by the following two conditions.
\begin{enumerate}
\item If $i_N=i$, then $\varsigma_{\mathbf{i},i}(x)={x_N}$.
\item We have
$[\varsigma_{\mathbf{i},i}]_{trop}\circ \Psi^{\mathbf{j}}_{\mathbf{i}}=[\varsigma_{\mathbf{j},i}]_{trop}.$
\end{enumerate}
\end{definition}

In \cite{GKS20} we have shown that the tropicalization of this functions gives rise to the string cone inequalities.

\begin{proposition}\cite[Proposition 3.5]{GKS20}\label{stringpos} For $\mathbf{i} \in R(w_0)$, we have
\begin{equation}\label{stringconeine}
\mathcal{C}_{\mathbf{i}}=\{x \in [\mathcal{T}_{\mathbf{i}}]_{trop} \mid [\varsigma_{\mathbf{i},i}]_{trop}(x)\ge 0 \text{ for all }i\in I.\}.
\end{equation}
\end{proposition}

The explicit form of the function $[\varsigma_{\ii,i}]_{trop}$ is not known. Explicit string cone inequalities are obtained in \cite{Lit} for a special class of reduced words called nice words (see Section \ref{nicesec} for the definition of nice words) and in \cite{GP} for all reduced words in type $A$ (also in \cite{BZ} for arbitrary reduced words but in a less explicit form). In \cite{GKS20} we show that the functions $[\varsigma_{\ii,i}]_{trop}$, $i\in I$ recover the string cone inequalities from \cite{GP} in type $A$. In Section \ref{nicesec} we further show that the functions $[\varsigma_{\ii,i}]_{trop}$, $i\in I$, recover the string cone inequalities from \cite{Lit}.

\subsection{(Non-)redundancy of string cone inequalities}
We fix $\mathbf{i} \in R(w_0)$ and $i\in I$. In this section we present a sufficient criterion for the non-redundancy of string cone inequalities (Theorem \ref{nomulti}). The missing proofs are provided in Section \ref{proofs}.

First we note that the function $\varsigma_{\mathbf{i},i}$ is a Laurent polynomial with non-negative integer coefficients. 

\begin{lemma} We have $\varsigma_{\mathbf{i},i}\in \mathbb{Z}_{\ge 0}[x_j^{\pm 1} \mid j \in N]$.
\end{lemma}
\begin{proof}
The fact that $\varsigma_{\mathbf{i},i}$ is a Laurent polynomial follows from \cite[Corollary 7.6.]{GKS20} and the positivity of the coefficients is guaranteed by the definition of $\varsigma_{\mathbf{i},i}$.
\end{proof}

We denote by $\mathcal{M}_{j,\ii}(\varsigma)$ the set of Laurent monomials of $\varsigma_{\mathbf{i},i}$. Our first observation is that the inequalities arising from $\varsigma_{\mathbf{i},j}$, $j\ne i$, cannot be expressed in terms of the inequalities of $\varsigma_{\mathbf{i},i}$.

\begin{proposition}\label{stringred} Let $m_0\in \mathcal{M}_{i,\ii}(\varsigma)$ and assume that the inequality arising from $m_0$ is redundant, i.e. 
$$[m_0]_{trop}=\sum_{j\in I}\sum_{\mathfrak{m}\in \mathcal{M}_{ j ,\ii}(\varsigma)}r_{\mathfrak{m}}[\mathfrak{m}]_{trop}$$ 
with $r_{\mathfrak{m}}\in \mathbb{R}_+$. Then $\mathfrak{m}\in \mathcal{M}_{j,\ii}(\varsigma)$ with $i\ne j$ implies that $r_{\mathfrak{m}}=0$.
\end{proposition}

We define the following notion.

\begin{definition} Let $\mathbf{i}\in R(w_0)$ and $i\in I$. We say that $\varsigma_{\mathbf{i},i}=\sum_{k\in \mathbb{Z}^N}a_k x_1^{k_1}\cdots x_N^{k_N}$ is \emph{multiplicity-free} if for all $k\in \mathbb{Z}^N$ such that $a_k\ne 0$ we have $|k_j|\le 1$ for all $j\in [N]$.
\end{definition}

The notion of multiplicity-free gives us a sufficient criterion for the non-redundancy of string cones inequalities as the following theorem shows.

\begin{theorem}\label{nomulti} If $\varsigma_{\mathbf{i},i}$ is multiplicity-free, then the set of inequalities given in \eqref{stringconeine} is non-redundant.
\end{theorem}

We conjecture that the sufficient criterion of Theorem \ref{nomulti} is indeed also necessary.

\begin{conjecture}\label{conjmu2} The set of inequalities given in \eqref{stringconeine} is non-redundant if and only if $\varsigma_{\mathbf{i},i}$ is multiplicity-free.
\end{conjecture}

In Section \ref{exnonm} we provide an of Conjecture \ref{conjmu2} in a case which is not multiplicity-free.

\section{Cluster varieties and potential functions}
We first define the notion of a seed.
\begin{definition}\label{clusterquiver} Let $M$ be a finite index set and $M_0\subset M$. We associate a quiver  $\Gamma_{\Sigma}$ to a datum $\Sigma=(\Lambda, \left<,\right>_{\Sigma}, \{e_k\}_{k\in M})$, where
\begin{itemize}
\item[(i)] $\Lambda$ is a lattice,
\item[(ii)] $\left<,\right>_{\Sigma}$ is a skew-symmetric $\mathbb{Z}$-valued bilinear form on $\Lambda$,
\item[(iii)] $\{{e_k}\}_{k\in M}$ is a basis of $\Lambda$.
\end{itemize}
The set of vertices $\{v_k\}_{k\in M}$ of the quiver $\Gamma_{\Sigma}$ is indexed by the set $M$. The set $\{v_k\}_{k\in M_0}$ is called the subset of frozen vertices and the set $\{v_k\}_{k\in M\setminus M_0}$ is called the subset of mutable vertices. Two vertices $v_{k}$ and $v_{\ell}$ of $\Gamma_{\Sigma}$ are connected by $\left< e_k,e_{\ell}\right>_{\Sigma}$ arrows in $\Gamma_{\Sigma}$ with source $v_k$ and target $v_{\ell}$ if and only if $\left<e_k,e_{\ell}\right> \ge 0$. The datum $\Sigma$ is called a $\emph{seed}$.
\end{definition}
Let $\Sigma=(\Lambda, \left<,\right>_{\Sigma}, \{e_k\}_{k\in M})$ be a seed. For each $k\in M \setminus M_0$ we define the seed $\mu_k(\Gamma_{\Sigma})$, called the \emph{mutation of $\Sigma$ at $k$}, by replacing the basis $\{e_k\}_{k\in M}$ by the new basis $\{e'_k\}_{k\in M}$ defined as
\begin{equation*}
e'_j=\begin{cases} e_j+\max\{0, \left<e_j,e_k\right>_{\Sigma}\}e_k & \text{ if }j\ne k \\
-e_k & \text{ if }j=k.
\end{cases}
\end{equation*}

The quiver $\Gamma_{\mu_k (\Sigma)}=\mu_k(\Gamma_{\Sigma})$ has the same vertex set as $\Gamma_{\Sigma}$. The mutable vertices of $\mu_k(\Gamma_{\Sigma})$ equal the mutable vertices of $\Gamma_{\Sigma}$. The arrow set of $\mu_k(\Gamma_{\Sigma})$ equals the arrow set of $\Gamma_{\Sigma}$ with the following changes: 
	\begin{itemize}
		\item[(i)] All arrows of $\Gamma_{\Sigma}$ with source or target $v_k$ gets replaced in $\mu_k(\Gamma_{\Sigma})$ by the reversed arrow.
		\item[(ii)] For every pair of arrows $(h_1, h_2)\in \Gamma_{\Sigma}\times \Gamma_{\Sigma}$ with 
		$$v_k=\text{target of $h_1$}=\text{source of $h_2$}$$
		we add to $\mu_k(\Gamma_{\Sigma})$ an arrow with source the source of $h_1$ and target the target of $h_2$.
		\item[(iii)] If a $2$-cycles was obtained during (i) and (ii), the arrows of this $2$-cycle get erased in $\mu_k(\Gamma_{\Sigma})$.
		\item[(iv)] Finally we erase all arrows between frozen vertices.
	\end{itemize}
The quiver $\mu_k(\Gamma_{\Sigma})$ is called the \emph{mutation of $\Gamma_{\Sigma}$ at $k$}.

To each seed $\Sigma=(\Lambda, \left<,\right>_{\Sigma}, \{e_k\}_{k\in M})$ we assign a collection of $\mathcal{A}$-cluster variables $\{A_k(\Sigma)\}_{k\in M}$ and $\mathcal{X}$-cluster variables $\{X_k(\Sigma)\}_{k \in M}$ and tori
$$\mathcal{A}_{\Sigma}:=\text{Spec}\mathbb{C}[A^{\pm 1}_k(\Sigma) \mid k\in M], \qquad {{\mathcal{X}}}_{\Sigma}:=\text{Spec}\mathbb{C}[X^{\pm 1}_k(\Sigma) \mid k\in M],$$
called the \emph{$\mathcal{A}$-cluster torus} and the \emph{$\mathcal{X}$-cluster torus} associated to $\Sigma$, respectively. We call $\{A_k(\Sigma)\}_{k\in M_0}$ and $\{X_k(\Sigma)\}_{k \in M_0}$ the frozen $\mathcal{A}-$ and $\mathcal{X}$-cluster variables respectively.

Assume that the quiver $\Gamma_{\Sigma'}$ of the seed $\Sigma'$ is obtained from the exchange graph $\Gamma_{\Sigma}$ of the seed $\Sigma$ by mutation at the vertex $k$. We define birational transition maps (see \cite[Equations (13) and (14)]{FG})
\begin{align} \notag
{\mu_k^*} A_i(\Sigma') &   = \begin{cases}   {A_k^{-1}(\Sigma)  \left(\displaystyle\prod_{j \,:\, \left<e_j,e_k\right>_{\Sigma} >0} A_j(\Sigma) ^{\left<e_j,e_k\right>_{\Sigma}}   + \displaystyle\prod_{j \,:\, \left<e_j,e_k\right>_{\Sigma} < 0} A_j(\Sigma) ^{-\left<e_j,e_k\right>_{\Sigma}}\right)} &   \text{if $i=k$,} \\   A_i(\Sigma)  &   \text{else,} \end{cases} \\ \label{X-mutation}
 \check{\mu}_k^* X_i(\Sigma')  &   = \begin{cases}   X_k(\Sigma) ^{-1} &   \text{if $i=k$,} \\   X_i(\Sigma)  \left(1+X_k(\Sigma) ^{-\text{sgn} \left<e_i,e_k\right>_{\Sigma }}\right)^{-\left<e_i,e_k\right>_{\Sigma}} &   \text{else,} \end{cases}
\end{align}
which we call \emph{$\mathcal{A}$- and $\mathcal{X}$-cluster mutation}, respectively.

\begin{definition} Given a fixed initial seed $\Sigma_0$ the $\mathcal{A}$- and $\mathcal{X}$-cluster variety, respectively, is defined as the scheme 
$$\mathcal{A}:=\displaystyle\bigcup_{\Sigma} \mathcal{A}_{\Sigma} \qquad \mathcal{X}=\displaystyle\bigcup_{\Sigma} \mathcal{X}_{\Sigma}, $$
obtained by gluing the tori $\mathcal{A}_{\Sigma}$ and $\mathcal{X}_{\Sigma}$ along the $\mathcal{A}$- and $\mathcal{X}$-cluster mutation, respectively. Here $\Sigma$ varies over all seeds which can be obtained from $\Sigma_0$ by a finite sequence of mutations.

Furthermore we define a partial compactification $\overline{\mathcal{A}}$ of $\mathcal{A}$ by adding the divisors corresponding to the vanishing locus of the frozen cluster variables.
\end{definition}

In \cite{GHKK} a Landau-Ginzburg potential $W$ on $\mathcal{X}$ is defined as the sum of certain global monomials attached to the frozen cluster variables. We only give the definition in the case that every frozen cluster variable has an optimized seed.

\begin{definition}\label{potdefi} Let $\Sigma$ be a seed and $k\in M_0$ a frozen vertex of $\Gamma_{\Sigma}$. We say that \emph{$\Sigma$ is optimized for $k$} if whenever $k'\in M\setminus M_0$ is adjacent to $k$, the source of the corresponding arrow is $k'$. 

Let $\mathcal{X}$ be the cluster variety obtained from the initial datum $\Sigma$. If for every $k\in M_0$ there exists a seed $\Sigma_k$ of $\mathcal{X}$ which is optimized for $v_k$, we define 
$$W=\displaystyle\sum_{k\in M_0} W_k \in \mathbb{C}[{\mathcal{X}}]$$
by 
$$\restr{W_k}{{{\mathcal{X}}}_{\Sigma_k}}=X_k^{-1}(\Sigma_k).$$
\end{definition}

The following consequence of Fomin-Zelevinsky's separation formula is crucial for our approach.

\begin{proposition}\label{polynomial} Let $\Sigma_0=(\Lambda, \left<,\right>, \{e_k\}_{k\in M})$ be a seed and $\mathcal{X}$ be the cluster variety obtained from the initial datum $\Sigma_0$. Assume furthermore that every frozen vertex has an optimized seed. For any seed $\Sigma$ of $\mathcal{X}$ we have 
$$\restr{W}{{\mathcal{X}_{\Sigma}}}\in \mathbb{C}[X_k^{-1}(\Sigma) \mid k \in M],$$
\end{proposition}
\begin{proof} 
We proof the statement for an arbitrary $k\in M_0$. Let $\Sigma_k$ be a seed of $\mathcal{X}$ which is optimized for $k$. Assume that $\Sigma_k$ can be obtained from $\Sigma_0$ by the sequence of mutations at $v_{j_1},v_{j_2},\ldots,v_{j_m}$ and assume that $\Sigma$ can be obtained from $\Sigma_0$ by the sequence of mutations at $v_{\ell_1},v_{\ell_2},\ldots,v_{\ell_p}$.

Let $\tilde{\Sigma}_0=(\tilde{\Lambda}, \left<,\right>_{\tilde{\Sigma}_0}, \{e_k\}_{k\in \tilde{M}})$ where $\tilde{M}=(M\setminus M_0) \cup k$, $\tilde{\Lambda}$ is the sublattice of $\Lambda$ spanned by $e_j$ with $j\in \tilde{M}$ and $\left<,\right>_{\tilde{\Sigma}_0}$ is the restriction of $\left<,\right>_{\Sigma_0}$. We define all vertices $v_j$ of $\Gamma_{\tilde{\Sigma}}$ to be mutable. Hence the cluster variety $\tilde{\mathcal{X}}$ obtained from this initial datum has no frozen cluster variables. Moreover we denote by $\tilde{\Sigma}$ the seed of $\tilde{\mathcal{X}}$ which is obtained from the initial seed $\tilde{\Sigma}_0$ by mutations at the sequence $v_{\ell_1},v_{\ell_2},\ldots,v_{\ell_p}$ and be $\tilde{\Sigma_k}$ the seed of $\tilde{\mathcal{X}}$ which is obtained from the initial seed $\tilde{\Sigma}_0$ by mutations at the sequence $v_{\ell_1},v_{\ell_2},\ldots,v_{\ell_p}$. In other words $\tilde{\Sigma}$ ($\tilde{\Sigma_k}$) are obtained from the same sequence of mutation at vertices $v_j$, $j\in \tilde{M}\subset M$ as $\Sigma$ ($\Sigma_k$, respectively,) is obtained from $\Sigma_0$.

We define $Y(\tilde{\Sigma})=\{y_i(\tilde{\Sigma}) \mid i \in \tilde{M}\}$ to be the $Y$-variables of the corresponding cluster algebra of geometric type as defined in \cite[Section 5.7]{Ke13}. Comparing the mutation rule of $Y$-variables (see \cite[Equation 22]{Ke13}) with the one for $\mathcal{X}$-cluster variables given in \eqref{X-mutation}, we get by Fomin-Zelevinsky's separation formula (\cite{FZ2}, see also \cite[Theorem 5.7]{Ke13}) and the definition of $W_k$ (Definition \ref{potdefi})
 \begin{align} \notag &  \qquad \qquad \qquad \qquad \displaystyle\restr{W_k}{\mathcal{X}_{{\Sigma}}}= \displaystyle\restr{W_k}{\mathcal{X}_{\tilde{\Sigma}}}= \\ \label{Fpol} 
 &  \prod_{i\in \tilde{M}}X_i(\Sigma)^{-c_{i,k}(\tilde{\Sigma}_k)}\prod_{j\in \tilde{M}}F_k(\Sigma)(X^{-1}_1(\Sigma),\ldots,X^{-1}_n(\Sigma))^{\left<e_i,e_k\right>_{\tilde\Sigma_k}}.\end{align}
Here $F_k(\tilde{\Sigma})$ is the $F-$polynomial at the seed $\tilde{\Sigma}$ as defined in \cite[Section 5.4]{Ke13} and $(c_{1,k}\ldots c_{n,k})^{\text{tr}}$ is the $c-$vector as defined in \cite[Section 5.3]{Ke13}. Since the $F-$polynomial is a honest polynomial and $\left<e_i,e_k\right>_{\tilde\Sigma_k}\ge 0$ for all $i$ due to the fact that $\Sigma_k$ is optimized for the vertex $v_k$, it remains to show that $c_{i,k}\ge 0$ for all $i$. This follows from the fact that in the sequence of mutations from $\tilde{\Sigma}$ to $\tilde{\Sigma}_k$ we have never mutated at $k$.
\end{proof}

\section{String cones and potential functions}

Let $G$ be a simply connected complex semisimple Lie group with Lie algebra $\mathfrak{g}$. Let $H$ be a maximal torus with Lie algebra $\mathfrak{h}$. Let $B,B_{-}$ denote a pair of opposite Borel subgroups of $G$ with $B\cap B_{-}=H$ and $N\subset B$ be the unipotent radical of $B$. The \emph{reduced double Bruhat cell $L^{e,w_0}$ associated to $e$ and $w_0$} is defined as:
$$L^{e,w_0}=N \cap B_{-}w_0 B_{-}.$$

Following Berenstein-Fomin-Zelevinsky, Fomin-Zelevinsky and Fock-Goncharov \cite{FZ, BFZ2, FZ2, FG} we endow $L^{e,w_0}$ with the structure of a cluster variety. 

\begin{definition}\label{clustergraph} Following \cite{BFZ2} we associate to every reduced word $\ii=(i_1,i_2,\ldots,i_N)\in R(w_0)$ a seed $\Sigma_{\ii}$ of $L^{e,w_0}$ and the corresponding quiver $\Gamma_{\ii}=\Gamma_{\Sigma_{\ii}}$ with vertex set $\{v_k \mid k \in N\}$. For an index $k\in N$ we denote by $k^+=k^+_{\ii}$ the smallest index $\ell \in M$ such that $k<\ell$ and $i_{\ell}=i_k$. If no such $\ell$ exists, we set $k^+=N+1$. Two vertices $v_k,v_{\ell}$, $k<\ell$ are connected by an edge in $\Gamma_{\ii}$ if and only if $\{k^+,\ell^+\}\cup [N]\ne \emptyset$ and one of the two conditions are satisfied
\begin{itemize}
\item[(i)] $\ell=k^+$,
\item[(ii)] $\ell < k^+ <\ell^+$.
\end{itemize}
An edge of type $(i)$ is directed from $k$ to $\ell$ and an edge of type $(ii)$ is directed from $\ell$ to $k$. The set of frozen vertices of $\Gamma_{\ii}$ is given by all $v_k$ such that $k^+=N+1$.
\end{definition}

We recall the relation between seeds associated to different reduced words. 

\begin{lemma}[{\cite[Lemma 4.6]{GKS20}}]\label{braid1} Let $\mathbf{j}\in R(w_0)$ be obtained from $\ii\in {R}(w_0)$ by a $3$-term move in position $k$. Then the swapping of vertex $v_k$ with vertex $v_{k+1}$ is an isomorphism of quivers $\Gamma_{\mathbf{j}}\cong\mu_{k-1} \Gamma_{\ii}.$
\end{lemma}

From Definition \ref{clustergraph} of a seed associated to a reduced word and Lemma \ref{braid1} we get immediately the following lemma.

\begin{lemma}\label{opsequence} Let $\ii=(i_1,\ldots,i_N)\in R(w_0)$. Then $\Gamma_{\ii}$ is optimized for the frozen vertex $v_N$. Moreover, for any $\mathbf{j}\in R(w_0)$ and any frozen cluster variable there exists a sequence of mutations to a seed which is optimized for this cluster variable. All seeds appearing in this sequence are of the form $\Gamma_{\ii'}$ for some $\ii'\in R(w_0)$.
\end{lemma}

We fix a reduced word $\ii\in R(w_0).$ In the following we explain the relation between the string cone inequalities and the potential function on $\mathcal{X}$ with initial datum given by Definition \ref{clustergraph}. 

\begin{definition} We define $\widehat{\text{CA}}_{\ii}\in \text{Hom} ((\mathbb{C}^*)^{\Delta^+_{\mathbf{i}}}, {\mathcal{X}}_{\Sigma_{\ii}})$ as follows
$$\widehat{\text{CA}}_{\ii}(x)_k=\displaystyle\prod_{\ell\in [N]}x_{\ell}^{\sym{k}{\ell}},$$
where $\sym{k}{\ell}:=-c_{i_k,i_{\ell}}\begin{cases} 1 & \text {if }k<\ell<k^+, \\ \frac{1}{2} & \text{if }\ell=k \text{ or }\ell=k^+, \\ 0 & \text{ else.} \end{cases}$
\end{definition}

\begin{proposition}\label{iso} The map $\widehat{\text{CA}}_{\ii}\in \text{Hom} ((\mathbb{C}^*)^{\Delta^+_{\mathbf{i}}}, \mathcal{X}_{\Sigma_{\ii}})$ is an isomorphism of algebraic tori and satisfies
$\varsigma_{\mathbf{i},i}=\restr{W_i}{\mathcal{X}_{\ii}}\circ \widehat{\text{CA}}_{\ii}$.
\end{proposition}
\begin{proof} 
The second part is proved analogously as \cite[Theorem 7.5]{GKS20} noting that the right diagram in Lemma 7.4 of op. cit. also commutes if we specialize the variables $X_k$ of $\mathcal{X}_{\Sigma_{\ii}}$ with $k<0$ and the variables $\lambda_j$, $j\in I$ of $\text{gr}\mathcal{S}_{\ii}$ to $1$.

The first part follows by Lemma 8.1 of op. cit. applying the same specialization of variables.
\end{proof}

\section{Proof of the sufficient criterion for non-redundancy of inequalities}\label{proofs}

Let $\mathcal{M}_{i,\ii}$ be the set of monomials in $\restr{W_i}{{\mathcal{X}}_{\ii}}$. We abbreviate the cluster variable $X_k({\Sigma_{\ii}})$ by $X_k(\ii)$. For a fixed $i\in I$ we denote a vertex $v_k$ of $\Gamma_{\ii}$ by $v_{i,\ell}$ if $k\in [N]$ and $\beta_{i,\ell}=\beta_k \in \Delta^+_{\ii}$. If we want to stress that we fix the reduced word $\ii \in R(w_0)$, we also write $v_{i,\ell}(\ii)$.
	
\begin{lemma}\label{prepare} \begin{enumerate}
\item For every $\mathfrak{m} \in \mathcal{M}_{i,\ii}$, we have $\mathfrak{m}=X^{-1}_{i,m_i}(\ii)\mathfrak{m}'$, where $\mathfrak{m}'$ is a Laurent monomial in which $X^{-1}_{\ii,m_i}$ does not appear with a strictly positive exponent.
\item Let $j\in I$ and $m_0\in \mathcal{M}_{j,\ii}$. Assume that the inequality arising from $m_0$ is redundant, i.e. 
$$[m_0]_{trop}=\sum_{i\in I}\sum_{\mathfrak{m}\in \mathcal{M}_{ i ,\ii}}r_{\mathfrak{m}}[\mathfrak{m}]_{trop}$$ 
with $r_{\mathfrak{m}}\in \mathbb{R}_+$. Then $\mathfrak{m}\in \mathcal{M}_{i,\ii}$ with $i\ne j$ implies that $r_{\mathfrak{m}}=0$.
\end{enumerate}
\end{lemma}
\begin{proof} Let $\ii'=(i'_1,\ldots,i'_N)\in R(w_0)$ be such that $i'_N=i$. Since $\ii'$ can be obtained from $\ii$ by a finite sequence of $2$-term and $3$-term moves, we can find by Lemma \ref{braid1} a sequence of mutations at vertices $(k_1,\ldots ,k_t)$ corresponding to these braid moves, transforming (up to relabeling coordinates) $\Sigma_{\ii}$ to $\Sigma_{\ii'}$. By Definition \ref{potdefi} and Lemma \ref{opsequence} we have
\begin{equation}\label{eq:movesop}
\restr{W_i}{{\mathcal{X}}_{\ii}}= \check{\mu}^*_{k_1} \circ \ldots \circ \check{\mu}^*_{k_t} X^{-1}_{i,m_i}(\ii').
\end{equation}
  
We prove $(1)$ by induction over $t$. The case $t=0$ is trivial. Let $\mathbf j$ be the reduced word we get after applying the sequence of mutations at the vertices $(k_{t-2},\ldots ,k_{1})$ to $\ii'$. We have
$$\restr{W_i}{{\mathcal{X}}_{\ii}}=\check{\mu}^*_{k_1} \sum_{i\in I}\sum_{\mathfrak{m}\in \mathcal{M}_{i,\mathbf{j}}} \mathfrak{m}.$$
By induction hypothesis, we have for every $\mathfrak{m}\in \mathcal{M}_{i,\mathbf j}$ that $\mathfrak{m}=X^{-1}_{i,m_i}(\mathbf{j})\mathfrak{m}'$, where $\mathfrak{m}'$ is a monomial in which $X^{-1}_{i,m_i}(\mathbf j)$ does not appear with a strictly positive exponent. Thus 
$$\check{\mu}^*_{k_1} \mathfrak{m}=\check{\mu}^*_{k_1}  X^{-1}_{i, m_i}(\mathbf j) \mathfrak{m}'.$$ 
Since $X_{i, m_i}(\mathbf{j})$ is frozen and does not appear with strictly positive exponent in $\mathfrak{m}'$, the cluster variable $X_{i, m_i}(\mathbf{i})$ does not appear with strictly positive exponent in any monomial of $\check{\mu}^*_{k_1}  \mathfrak{m}'$. 

Let $\ell \in N$ be such that $X_{i,m_i}(\mathbf j)=X_{\ell}(\mathbf{j})$. Then we have (again since $X_{i,m_i}(\mathbf j)$ is frozen)
\begin{equation*}
\check{\mu}^*_{k_1} X^{-1}_{ \ell}(\mathbf{j})= X_{\ell}^{-1}(\mathbf{i}) \left(1+X_k(\ii)^{-\text{sgn} \left<e_{\ell},e_k\right>_{\ii}}\right)^{-\left<e_{\ell},e_k\right>_{\ii}}.  
\end{equation*}
The first claim follows. 

To prove $(2)$ note that by the $\mathcal{X}$-cluster mutation rule \eqref{X-mutation} and \eqref{eq:movesop} we get that a variable $X_{\ell}$ divides any monomial of $\check{\mu}^*_{k_s} \circ \ldots \circ \check{\mu}^*_{k_t} X^{-1}_{i,m_i}(\ii)$ for $1\le s \le t$ if and only if either $X_{\ell}$ divides a monomial of $\check{\mu}^*_{k_{s+1}} \circ \ldots \circ \check{\mu}^*_{k_t} X^{-1}_{i,m_i}(\ii)$ or $s=\ell$. 

This implies in particular that $X_{i,m_i}(\ii)$ does not divide any element of $\mathcal{M}_{j,\ii}$.

On the other hand, we have have, by the first part of this lemma, that $X_{i,m_i}(\ii)$ divides any element of $\mathcal{M}_{i,\ii}$ which proves the second part.

\end{proof}

\begin{proof}[Proof of Proposition \ref{stringred}] Assume that
$$[m_0]_{trop}=\sum_{j\in I}\sum_{\mathfrak{m}\in \mathcal{M}_{ j ,\ii}(\varsigma)}r_{\mathfrak{m}}[\mathfrak{m}]_{trop}$$ 
with $r_{\mathfrak{m}}\in \mathbb{R}_+$. This implies by Proposition \ref{iso}

$$[ m_0 \circ \widehat{\text{CA}}^{-1}_{\ii}]_{trop}=\sum_{j\in I}\sum_{\mathfrak{m}\in \mathcal{M}_{ j ,\ii}(\varsigma)}r_{\mathfrak{m}}[\mathfrak{m} \circ \widehat{\text{CA}}^{-1}_{\ii}]_{trop}.$$

Now Lemma \ref{prepare} yields the claim.
\end{proof}

\begin{definition}\label{mf} Let $i\in I$ and $\ii \in R(w_0)$. We say that $\restr{W_i}{\mathcal{X}_{\ii}}$ is \emph{multiplicity-free}, if for all $\mathfrak{m}\in \mathcal{M}_{i,\ii}$, $\mathfrak{m}=\prod_{k=1}^N a_k X_k^{j_k}$, we have $|j_k|\le 1$. 

We say that $W$ is multiplicity-free for $\ii$ if $W_i$ is multiplicity-free for $\ii$ for all $i\in I$.
\end{definition}
Note that by Proposition \ref{polynomial} $|j_k|\le 1$ is equivalent to $j_k \ge -1$ in Definition \ref{mf}.

\begin{proposition}\label{nomult} Let $i\in I$ and $\ii \in R(w_0)$ and $\restr{W_i}{\mathcal{X}_{\ii}}$ be multiplicity-free. Then the set of inequalities 
$$\{[\mathfrak{m}]_{trop} (x) \ge 0 \mid \mathfrak{m}\in \mathcal{M}_{i,\mathbf i}\}$$
is non-redundant.
\end{proposition}

\begin{proof}
Let $\restr{W_i}{\mathcal{X}_{\ii}}$ be multiplicity-free and assume that the inequality arising from $m_0\in \mathcal{M}_{i,\ii}$ is redundant, i.e. there exists $\emptyset\ne J\subset \mathcal{M}_{i,\mathbf i}\setminus \{\mathfrak{m}_0\}$ such that
\begin{equation}\label{equ}
[\mathfrak{m}_0]_{trop}=\sum_{\mathfrak{m}\in J}r_{\mathfrak{m}}[\mathfrak{m}]_{trop}
\end{equation}
with $r_{\mathfrak{m}}\in \mathbb{R}_{>0}$. Let $k$ be such that $i_k=i$ and $i_{k}^+=N+1$ and let $e_{k}\in \mathbb{R}^N$ be defined as $(e_k)_j=\delta_{k,j}$. By the first part of Lemma \ref{prepare} and the fact that $\restr{W_i}{\mathcal{X}_{\ii}}$ is multiplicity-free, we get by plugging in $-e_k$ into \eqref{equ}: 
$$1=\sum_{\mathfrak{m}\in J}r_{\mathfrak{m}}.$$
Since $\emptyset\ne J\subset \mathcal{M}_{i,\mathbf i}\setminus \{\mathfrak{m}_0\}$, we can find $\mathfrak{m}_0\ne \mathfrak{m_1}\in J$ and $1 \le s \le N$ such that either (1) $[\mathfrak{m}_0](e_s)\ne 0$ and $[\mathfrak{m}_1](e_s)= 0$ or (2) $[\mathfrak{m}_0](e_s)=0$ and $[\mathfrak{m}_1](e_s)\ne 0$.

In the first case we get by Proposition \ref{polynomial} and the assumption that $\restr{W_i}{\mathcal{X}_{\ii}}$ is multiplicity-free by plugging in $-e_s$ into \eqref{equ}: 

$$0=1 + \sum_{\mathfrak{m}\in J\setminus{\mathfrak{m}_1}}r_{\mathfrak{m}}[\mathfrak{m}]_{trop}(-e_s).$$

Since $[\mathfrak{m}]_{trop}(-e_s)\ge 0$ by Proposition \ref{polynomial}, we obtain a contradiction.

In the second case, we get by by plugging in $-e_s$ into \eqref{equ} again using Proposition \ref{polynomial} and the assumption that $\restr{W_i}{\mathcal{X}_{\ii}}$ is multiplicity-free:

$$1= \sum_{\mathfrak{m}\in J\setminus{\mathfrak{m}_1}}r_{\mathfrak{m}}[\mathfrak{m}]_{trop}(-e_s).$$

By Proposition \ref{polynomial} and the assumption that $\restr{W_i}{\mathcal{X}_{\ii}}$ is multiplicity-free we have $0 \le [\mathfrak{m}]_{trop}(-e_s)\le 1$, and hence 

$$\sum_{\mathfrak{m}\in J\setminus{\mathfrak{m}_1}}r_{\mathfrak{m}}[\mathfrak{m}]_{trop}(- e_s)< \sum_{\mathfrak{m}\in J}r_{\mathfrak{m}}=1.$$
Again a contradiction.
\end{proof}

\begin{proof}[Proof of Theorem \ref{nomulti}]
Note that, by Proposition \ref{iso}, $\varsigma_{\mathbf{i},i}$ is multiplicity-free if and only if $\restr{W_i}{\mathcal{X}_{\ii}}$ is multiplicity-free. Moreover, the set of inequalities given in \eqref{stringconeine} is redundant if and only if the set of inequalities $\{[\mathfrak{m}]_{trop} (x) \ge 0 \mid \mathfrak{m}\in \mathcal{M}_{i,\mathbf i}\}$ is redundant. Therefore the claim follows from Proposition \ref{nomulti}.
\end{proof}

	\section{Multiplicity-free examples}

\subsection{Nice words and simply-braided words}\label{nicesec}
\subsubsection{Simply braided words}
 
Following \cite{SST} we define
	\begin{definition}\label{simply-braided} Let $i \in I$. We call $\ii\in R(w_0)$ simply-braided for $i$ if one can perform a sequence of braid moves changing $\ii$ to a reduced word $\ii'=(i'_1,\ldots,i'_N)\in R(w_0)$ with $i'_N=i$, and each move in the sequence is either \begin{itemize}
			\item a $2-$move, or
			\item a $3-$ move at position $k$ such that $\alpha_i=\beta_{k-1}$ in the $<_{\ii}$-order on $\Phi^+$, i.e. $\alpha_i$ is the leftmost root affected.
		\end{itemize}
		We call $\ii$ simply-braided if it is simply-braided for all $i\in I$. We call the above sequence from $\ii$ to $\ii'$ of $2$- and $3-$moves a simply-braided sequence.
		\end{definition}
		
\begin{remark} Note that we take a slightly different convention than \cite{SST} here. If $\ii$ satisfies Definition \ref{simply-braided} this implies that the reverse word is simply-braided in the sense of \cite{SST}.
\end{remark}

\begin{proposition}\label{tubes} Let $\ii\in R(w_0)$ be simply-braided $i$ and fix a simply-braided sequence. Let $\beta_{k_1},\ldots, \beta_{k_s}\in \Delta^+$ be indexed w.r.t. the $<_{\ii}$-order such that in the $j$-th $3-$term move in the simply braided sequence $\beta_{k_j}$ is the middle root affected. Then
			$$\restr{W_i}{{{\mathcal{X}}_{\ii}}}=X_{\alpha_i}^{-1}(\ii)(\sum_{\ell =1}^{s}\prod_{j=1}^{\ell}X_{k_j}^{-1}(\ii)+1).$$
		\end{proposition}

		\begin{proof} We prove the claim by induction over $m$. If $m=0$, then $\Gamma_{\ii}$ is optimized for $i$, $W_i|_{\mathbf{{\mathcal{X}}}_{\ii}}=X_{\alpha_i}^{-1}$ and our claim is true. Assume now that $m>0$ and let $\beta_t=\alpha_k$. Since $\ii$ is by assumption simply-braided, we get (up to $2$-term moves) that $(i_t,i_{t+1},i_t+2)$ is the first $3-$term move in the simply-braided sequence. Let $\ii'$ be the reduced word obtained from $\ii$ by a $3$-term move at position $t+1$. Clearly $\ii'$ is still simply-braided with a simply-braided sequence such that $\beta_{k_j}$ is the middle root affected in the $j+1-$th $3$-term move for all $j\in\{2,\ldots,s\}$. Hence, by induction hypothesis,
		$$\restr{W_i}{_{\mathcal{X}_{\ii'}}}=\mu^{\vee}_{t} \circ \restr{W_i}{{\mathcal{X}_{\ii}}}=X_{\alpha_i}^{-1}(\ii')(\sum_{\ell =1}^{s}\prod_{j=2}^{\ell}X_{k_j}^{-1}(\ii')+1).$$
		
		By the definition of the graph $\Gamma_{\ii'}$, it looks locally around $v_{t}$ as follows:
		\begin{center}
		$
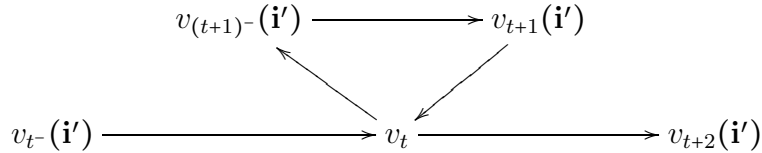
\xymatrix{
		&  v_{({t+1})^{-}}(\ii') \ar[rr] & & v_{t+1}(\ii') \ar[dl] \\
		v_{t^{-}}(\ii') \ar[rr]& & v_t \ar[rr] \ar[lu]& & v_{t+2}(\ii')
		}$
		\end{center}
		Since $\beta_{k_j}\le_{\ii} \alpha_i$ for all $j<m$ and $X_{t+2}(\ii')=X_{\alpha_i}(\ii')$, $X_{t}(\ii')=X_{k_m}(\ii')$, we get the claim by the $\mathcal{X}$-cluster mutation rule \eqref{X-mutation} and Lemma \ref{braid1}.
		\end{proof}
		
		We are ready to prove that the string cone inequalities are non-redundant for simply-braided $\ii$.
		
		\begin{theorem}\label{simbraid} If $\ii\in R(w_0)$ is simply-braided, $\varsigma_{\ii,i}$ is multiplicity-free. In particular, the inequalities from \eqref{stringconeine} are non-redundant.
		\end{theorem}
		
		\begin{proof} We have for all $i\in I$ that $\restr{W_i}{\mathcal{X}_{\ii}}$ be multiplicity-free by Proposition \ref{tubes}. Hence the the inequalities are non-redundant by Propsition \ref{iso} and Theorem \ref{nomult}.
		\end{proof}

\subsubsection{Nice words}\label{nicesec}

We recall the following notions from \cite{Lit}.

\begin{definition}\label{nice} \begin{enumerate}
\item We call an enumeration $\{\alpha_1, \ldots, \alpha_n\}$ of $\Pi$ a \emph{good enumeration} if for all $i\in \{1,\ldots,n\}$ the fundamental weight $\omega_i$ is minuscule for $G'$ corresponding to the Dynkin diagram of $G$ with the nodes labeled by $\{1,\ldots,i-1\}$ removed. 
\item Let $\{\alpha_1, \ldots, \alpha_n\}$ be a good enumeration. For $j\in \{1,\ldots, n\}$ we denote by $W_j$ the subgroup of $W$ generated by $s_1,\ldots,s_j$. We call a reduced word $\ii=(i_1,\ldots,i_N)\in R(w_0)$ a \emph{nice word} if  $s_{i_1}s_{i_2}\cdots s_{i_N}=\tau_1 \tau_2 \cdots \tau_n$, where $\tau_j$ is the longest word in the set of minimal representative in $W_j$ of $\mfaktor{W_{j-1}}{W_j}$.
\end{enumerate}
\end{definition}

\begin{remark} Nice words exist for all $\mathfrak{g}$ not of type $E_8$ by \cite{Lit}. Moreover, each $\tau_j$ in Definition \ref{nice} is unique up to $2-$term braid moves by \cite[Lemma 3.2]{Lit}.
\end{remark}

By \cite[Lemma 4.17]{SST} every nice word $\ii\in R(w_0)$ is simply-braided for all $i\in I$. Hence we get as a direct consequence of Proposition \ref{tubes} and Proposition \ref{iso}.
\begin{theorem}\label{nicew} Let $\ii \in R(w_0)$ be a nice word. Then $\restr{W_i}{\mathcal{X}_{\ii}}$ is multiplicity-free and the inequalities of $\mathcal{C}_{\ii}$ from \ref{stringconeine} are non-redundant. Moreover, the inequalities \ref{stringconeine} are explicitly given by $t_k \ge t_{k'}$ for any $k,k'\in [N]$ such that $\beta_{k}\le_{\ii} \beta_{k'}$ and thus recover the inequalities from \cite{Lit}.
\end{theorem}

\subsection{Minuscule weights and $\mathbf{i}$-trails}\label{trails}

We recall the notion of an $\ii$-trail from \cite{BZ}.

\begin{definition} For a finite dimensional representation $V$ of $\mathfrak{g}$, two weights $\gamma,\delta$ of $V$ and $\ii \in R(w_0)$, we say that a sequence of weights $\pi=(\gamma=\gamma_0,\gamma_1, \ldots, \gamma_{N}=\delta)$ is an $\ii$-trail from $\gamma$ to $\delta$ if
\begin{itemize}
\item $\gamma_{k-1}-\gamma_k=c_k \alpha_{i_k}$ for some $c_k\in \mathbb{Z}_{\ge 0}$,
\item $e_{i_1}^{c_1}e_{i_2}^{c_2}\cdots e_{i_{\ell}}^{c_{\ell}}$ is a non-zero map vom $V_{\delta}$ to $V_{\gamma}$. Here $V_{\delta}$ ($V_{\gamma}$) denotes the weight space of $V$ corresponding to the weight $\delta$ ($\gamma$, respectively).
\end{itemize}
We further define for any $\ii$-trail $\pi=(\gamma_0,\gamma_1, \ldots, \gamma_{\ell})$ in a $\mathfrak{g}$-module and every $k\in [\ell]$ the value
$$d_k=d_k(\pi)=\frac{\gamma_{k-1}+\gamma_k}{2}(\alpha_{i_k}^{\vee}).$$
\end{definition}

By \cite[Theorem 3.10]{BZ} the string cone $\mathcal{C}_{\ii}$ is the cone in $\mathbb{R}^N$ given by all $(t_1,\ldots,t_N)$ such that $\sum_{k}d_k(\pi)t_k\ge 0$ for any $i\in I$ and any $\ii$-trail $\pi$ from $\omega^{\vee}_i$ to $w_0s_i\omega_i^{\vee}$ in the $^L\mathfrak{g}$-module $V(\omega^{\vee}_i)$.

We call a fundamental weight $\omega_i$ \emph{minuscule} if $\beta(\omega_i^{\vee}) \in \{-1,0,1\}$ for all $\beta\in \Delta$.  
The aim of this section is to prove that the defining inequalities of $\mathcal{C}_{\mathbf{i}}(i)$ are non-redundant provided $\omega_i$ is minuscule. The essential argument relies on extremal $\mathbf{i}$-trails introduced in \cite{BZ}. We first give the relation to our function $\varsigma_{\mathbf{i},i}$ from \eqref{stringconeine}. 

Recall for $i\in I$ that we denote by $i^*\in I$ the unique element such that $w_0\omega_i=\omega_{i^*}$. Note that $\omega_i$ is minuscule if and only if $\omega_{i^*}$ is minuscule.

\begin{proposition}\label{trailandpot} Let $i \in I$ be such that $\omega_i$ is minuscule and $\ii \in R(w_0)$. We have for $t\in \mathbb{R}^N$ 
$$[\varsigma_{\mathbf{i},i}]_{trop}(t)=\displaystyle\min_{\pi}\sum_{k}d_k(\pi)t_k\ge 0,$$ 
where the minimum is taken over all $\ii$-trails $\pi$ from $\omega^{\vee}_{i^*}$ to $w_0s_{i^*}\omega_{i^*}^{\vee}$ in $V(\omega^{\vee}_i)$.
\end{proposition}
\begin{proof} For $i\in I$ let us denote by $\mathfrak{t}_{i}: \mathbb{Z}^N \rightarrow \mathbb{Z}$ the piecewise-linear function such that $\mathfrak{t}_{\ii}(t)=\min_{\pi}  \sum_{k}d_k{\pi}t_k$ where the minimum is taken over all $\ii$-trails $\pi$ from  $\omega^{\vee}_{i^*}$ to $w_0s_{i^*}\omega_{i^*}^{\vee}$ in $V(\omega^{\vee}_i)$. To prove the claim we need to show that $[\varsigma_{\mathbf{i},i}]_{trop}=\mathfrak{t}_{i}$. By the proof of \cite[Theorem 3.10]{BZ}, we have or any $\ii' \in R(w_0)$. 
\begin{equation}\label{trailsmut}
 \mathfrak{t}_{\ii} \circ \Psi^{\ii'}_{\ii}=\mathfrak{t}_{\ii'}.
\end{equation}
By Definition \ref{varsigma} it suffices to prove the claim for one fixed $\ii \in R(w_0)$. Let $w'_0$ be the longest element in the maximal parabolic subgroup of $W$ generated by $\left<s_j \mid j\in I\setminus \{i\} \right>$. Let $\tau_i$ be as in Definition \ref{nice}. Let $\ii_0$ be a reduced word such that $\ii_0=(\ii'_0,\ii''_0)=(i'_1,\ldots,i'_N)$ where $\ii'_0$ is a reduced word for $w'_0$ and $\ii''_0$ is a reduced word for $\tau_i$. By \cite[Lemma 3.2]{Lit}, the word $\ii'_0$ is unique up to 2-term braid moves hence we are in the situation of \cite[Proof of Theorem 3.13]{BZ}. From this we conclude that $\mathfrak{t}_{\ii_0}(t)=t_N$ and that $i'_N=(i^*)^*=i$. We conclude, by Definition \ref{varsigma}, that $[\varsigma_{\mathbf{i}_0,i}]_{trop}(t)=t_N$ which proves the claim.
\end{proof}

We are ready to prove the main result of this section.

\begin{theorem}\label{minusculenomulti} Let $i \in I$ be such that $\omega_i$ is minuscule and $\ii \in R(w_0)$. Then $\varsigma_{\mathbf{i},i}$ is multiplicity-free and the inequalities from \eqref{stringconeine} are non-redundant.
\end{theorem}
\begin{proof} Let $W_{\hat{i}}$ be the maximal parabolic subgroup of $W$ generated by all $s_j$ with $j\ne i$ and let $u{(i)}$ be the minimal representative of the coset $W_{\hat{i}}s_iw_0$ in $W$. Let $\mathfrak{t}_{\ii}(t)$ be as in the proof of Proposition \ref{trailandpot}. With the convention that $k(0) = 0$ and $k(p+1) = N+1$ we have by \cite[Proposition 9.2, Theorem 3.10]{BZ} that
$$\mathfrak{t}_{\ii}(t)=\min_{(i_{k(1)},\ldots ,i_{k(p)})}\sum_{j=0}^{p}\sum_{k(j)<k<k(j+1)}s_{i_{k(1)}},\cdots s_{i_{k(j+1)}}\alpha_{i_k}(\omega_i^{\vee})\cdot t_k,$$ where the minimum ranges over all subwords $(i_{k(1)},\ldots i_{k(p)})$ of $\mathbf{i}$ which are a reduced word for $u(i)$. The claim now follows from the assumption that $\omega_i$ is minuscule and Proposition \ref{trailandpot}.
\end{proof}

We get as a direct corollary.

\begin{corollary}\label{typeA} Let $\mathfrak{g}=\text{sl}_{n+1}(\mathbb{C})$. Then $\varsigma_{\mathbf{i},i}$ is multiplicity-free and the inequalities from \eqref{stringconeine} are non-redundant for all $i\in I$.
\end{corollary}

Corollary \ref{typeA} was already proven in \cite[Proposition 4.5.]{CKLP}. Using the notation of \cite{CKLP} we note that by \cite{GKS21} 
$$\restr{W}{\mathcal{X}_{\ii}}=\displaystyle\sum_{P}\prod_{\ \ \mathcal{C}_j\text{ is in a chamber enclosed by }P}X_j(\ii)^{-u_j},$$
where $P$ varies over all rigorous paths.

We end this section by remarking that an algorithm for the computation of all $\ii$-trails from $\omega^{\vee}_i$ to $w_0s_i\omega_i^{\vee}$ for minuscule $\omega_i$ was given recently in \cite{KKN} by computing the monomials in the Berenstein-Kazhdan decoration functions as defined in op. cit. From this one may deduce an alternative proof of Theorem \ref{minusculenomulti} by combining \cite[Theorem 2.5 and the argument below, Lemma 2]{KKN} and \cite[Theorem 7.5]{GKS20}.

\section{Beyond the multiplicity-free case}\label{exnonm}

In this subsection we study the following example. Let $\mathfrak{g}=\text{so}_8(\mathbb{C})$. We fix a labelling of the Dynkin diagram as follows 
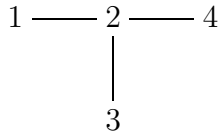
\begin{equation*}
\xymatrix{
1  \ar@{-}[r] &  2  \ar@{-}[r]  \ar@{-}[d]& 4 \\
& 3
}
\end{equation*}

Fix $\mathbf{i}=(2,1,4,2,3,2,4,2,1,2,3,4)\in R(w_0).$ The quiver $\Gamma_{\ii}$ looks as follows:

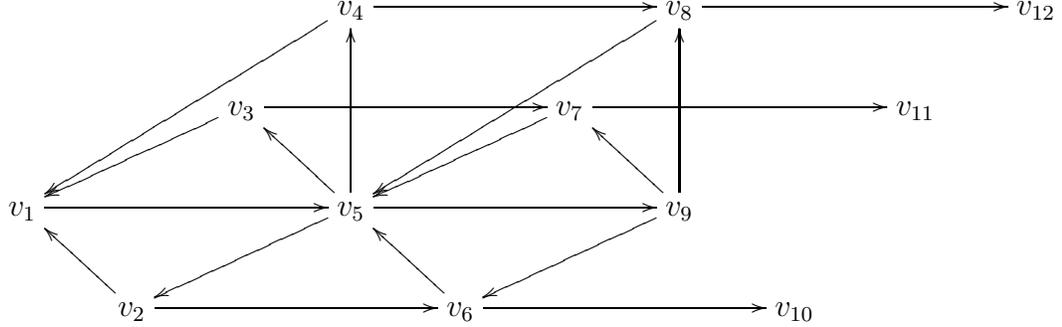
\begin{equation*}
\xymatrix{
& & & v_4 \ar[ddlll] \ar[rrr] & & & v_8 \ar[ddlll] \ar[rrr] & & & v_{12} \\
& & v_3 \ar[dll] \ar[rrr] & & & v_7\ar[dll]  \ar[rrr] & & & v_{11} \\
v_1\ar[rrr] & & & v_5 \ar[dll] \ar[ul] \ar[uu] \ar[rrr]  & & & v_9 \ar[ul] \ar[dll]\ar[uu] \\ & v_2 \ar[ul]\ar[rrr] & & & v_6 \ar[rrr] \ar[ul]& & & v_{10}
}
\end{equation*}

The vertices $v_{12}$, $v_{11}$, $v_9$ and $v_{10}$ are frozen. Let $\mathcal{X}$ be the cluster variety obtained from the initial datum $\Sigma_{\ii}$. Note that $\Sigma_{\ii}$ is optimized for the frozen vertices $v_{10}$, $v_{11}$ and $v_{12}$. Thus, by Definition \ref{potdefi}, we have
$$\restr{W_1}{\mathcal{X}_{\ii}}=X^{-1}_{10}(\ii), \quad \restr{W_3}{\mathcal{X}_{\ii}}=X^{-1}_{11}(\ii), \quad \restr{W_4}{\mathcal{X}_{\ii}}=X^{-1}_{12}(\ii).$$

It remains to compute $\restr{W_2}{\mathcal{X}_{\ii}}$. Note that $\omega_2$ is not minuscule and $\ii$ is not simply-braided for $2$. One checks that the following sequence of mutations leads to a seed $\Sigma'$ which is optimized for $v_{9}$:

$$(6,3,5,4,3,1,2,7,6,8).$$

We use this sequence to compute:

\begin{center}
\scalebox{1.0}{\parbox{\linewidth}{%
 $\restr{W_2}{\mathcal{X}_{\ii}} =  X^{-1}_9 X^{-1}_1  X^{-1}_2  X^{-1}_3  X^{-1}_4  X^{-2}_5  X^{-1}_6  X^{-1}_7  X^{-1}_8 + 
 X^{-1}_9 X^{-1}_2  X^{-1}_3  X^{-1}_4  X^{-2}_5  X^{-1}_6  X^{-1}_7  X^{-1}_8 +  
X^{-1}_9 X^{-1}_2  X^{-1}_3  X^{-2}_5  X^{-1}_6  X^{-1}_7  X^{-1}_8 + X^{-1}_9 X^{-1}_2  X^{-1}_4  X^{-2}_5  X^{-1}_6  X^{-1}_7  X^{-1}_8 + X^{-1}_9
X^{-1}_3  X^{-1}_4  X^{-2}_5  X^{-1}_6  X^{-1}_7  X^{-1}_8 +  X^{-1}_2  X^{-2}_5  X^{-1}_6  X^{-1}_7  X^{-1}_8 +  X^{-1}_9
  X^{-1}_3  X^{-2}_5  X^{-1}_6  X^{-1}_7  X^{-1}_8 + X^{-1}_9 X^{-1}_4  X^{-2}_5  X^{-1}_6  X^{-1}_7  X^{-1}_8 +  X^{-1}_9
 X^{-1}_2  X^{-1}_5  X^{-1}_6  X^{-1}_7  X^{-1}_8 + X^{-1}_9 X^{-1}_3  X^{-1}_5  X^{-1}_6  X^{-1}_7  X^{-1}_8 + X^{-1}_9 X^{-1}_4  X^{-1}_5  X^{-1}_6  X^{-1}_7  X^{-1}_8 +  X^{-1}_9 X^{-2}_5  X^{-1}_6  X^{-1}_7  X^{-1}_8 + X^{-1}_9 X^{-1}_4  X^{-1}_5  X^{-1}_6  X^{-1}_7 +  X^{-1}_9 X^{-1}_3  X^{-1}_5  X^{-1}_6  X^{-1}_8 +  X^{-1}_9 X^{-1}_2  X^{-1}_5  X^{-1}_7  X^{-1}_8 + 2 X^{-1}_9   X^{-1}_5  X^{-1}_6  X^{-1}_7  X^{-1}_8 + X^{-1}_9 X^{-1}_5  X^{-1}_6  X^{-1}_7 +  X^{-1}_9 X^{-1}_5  X^{-1}_6  X^{-1}_8 +  X^{-1}_9 X^{-1}_5  X^{-1}_7  X^{-1}_8 +  X^{-1}_9 X^{-1}_6  X^{-1}_7  X^{-1}_8 + X^{-1}_9  X^{-1}_6  X^{-1}_7 +  X^{-1}_9 X^{-1}_6  X^{-1}_8 +  X^{-1}_9 X^{-1}_7  X^{-1}_8 +  X^{-1}_9 X^{-1}_6 +  X^{-1}_7 +  X^{-1}_9 X^{-1}_8 + X^{-1}_9. $}}
\end{center}

One checks that the cone given by all $t\in \mathbb{R}^{12}$ such that $[\restr{W_2}{\mathcal{X}_{\ii}}]_{trop}(t)\ge 0$ has $26$ facets but $\restr{W_2}{\mathcal{X}_{\ii}}$ has $27$ monomials. Hence there must be a redundancy. This is expected by Conjecture \ref{conjmu2} since $\restr{W}{\mathcal{X}_{\ii}}$ is not multiplicity-free for $i=2$ by the above calculation.

Ones notes further that the redundancy is given by the following equality:

\begin{center}
$ \displaystyle[2 X^{-1}_9   X^{-1}_5  X^{-1}_6  X^{-1}_7  X^{-1}_8]_{trop} = \frac{1}{2} ([X^{-1}_9 X^{-2}_5  X^{-1}_6  X^{-1}_7  X^{-1}_8]_{trop}+ [X^{-1}_9 X^{-1}_6  X^{-1}_7  X^{-1}_8 ]_{trop}).$
\end{center}

Hence the redundant inequality corresponds precisely to the monomial with coefficient $2$. This inspires the following stronger version of Conjecture \ref{conjmu2}.

\begin{conjecture}\label{mult2} Let $i\in I$ and $\ii \in R(w_0)$. The inequality $[\mathfrak{m}]_{trop} (x) \ge 0$ for an $\mathfrak{m}\in \mathcal{M}_{i,\mathbf i}$ is redundant if and only if $\restr{W}{\mathcal{X}_{\ii}}$ is not multiplicity-free for $i$ and the coefficient $a \in \mathbb{Q}$ of $\mathfrak{m}$ satisfies $a>1$.
\end{conjecture}

\begin{remark}\label{ffei} Note that the coefficients of monomials of $\restr{W}{\mathcal{X}_{\ii}}$ are not visible in its tropicalization. Hence Conjecture \ref{mult2} suggests a criterion to determine facets of string cones which is not visible in the tropical version of the string cone inequalities. It would be very interesting to give a conceptual explanation of this. We suggest the following relation.

Recently Fei has showed in \cite{F19} for acyclic cluster algebras that the exponent vector of a monomial of the $F$-polynomial of any cluster variable gives rise to a vertex of its Newton polytope if and only if its coefficient is equal to $1$. He furthermore conjectures this to be true for any cluster algebras. Recall from the proof of Proposition \ref{polynomial} for $k\in I$
 \begin{equation*} \displaystyle\restr{W_k}{\mathcal{X}_{\ii}}=\prod_{i\in \tilde{M}}X_i(\ii)^{-c_{i,k}(\tilde{\Sigma}_k)}\prod_{j\in \tilde{M}}F_k(\Sigma_{\ii})(X^{-1}_1({\ii}),\ldots,X^{-1}_n(\ii))^{\left<e_j,e_k\right>_{\tilde\Sigma_k}},\end{equation*}
 where $\Sigma_k$ is an optimized seed. For our setup we may pick a seed $\Sigma_{\mathbf{j}}$ such that $\mathbf{j}\in W(w_0)$ and $j_N=k$. In this case the only arrow with target $v_k$ in $\Gamma_{\mathbf{j}}$ has source $v_{\ell}$ where $\ell^+=k$. Hence
  \begin{equation*} \displaystyle\restr{W_k}{\mathcal{X}_{\ii}}=X_{k}(\ii)^{-1}F_k(\Sigma_{\ii})(X^{-1}_1({\ii}),\ldots,X^{-1}_n(\ii)).\end{equation*}

Therefore the set of inequalities $[\restr{W_k}{\mathcal{X}_{\ii}}]_{trop}(t)\ge 0$ is non-redundant if and only if the exponent vector of every monomial of $F_k(\Sigma_{\ii})$ is a vertex of its Newton polytope. However, we are not aware of a relation between the exponents and the coefficients of the $F$-polynomial.

\end{remark}


	\end{document}